\documentclass[a4paper]{amsart}

\usepackage[T1]{fontenc}
\usepackage{amsmath, amssymb, amsthm, mathrsfs, mathtools, graphicx}

\usepackage{varioref}
\usepackage[backref=page]{hyperref}
\usepackage{enumitem, xspace, ifthen}
\usepackage[all]{xy}
\usepackage[nameinlink, capitalize]{cleveref}
\usepackage[dvipsnames]{xcolor}
\usepackage{tikz-cd}

\setlist[enumerate]{
  label=(\thethm.\arabic*),
  before={\setcounter{enumi}{\value{equation}}},
  after={\setcounter{equation}{\value{enumi}}},
  itemsep=1ex
}

\setlist[itemize]{
  leftmargin=*,
  topsep=1ex,
  itemsep=1ex,
  label=$\circ$
}

\newtheorem*{thm-plain}{Theorem}
\newtheorem{thm}{Theorem}[section]
\newtheorem{lem}[thm]{Lemma}

\newtheorem{cor}[thm]{Corollary}

\numberwithin{equation}{thm}

\newtheorem{bigthm}{Theorem}

\theoremstyle{definition}
\newtheorem{dfn}[thm]{Definition}

\newtheorem*{dfn-plain}{Definition}

\theoremstyle{remark}
\newtheorem{clm}[thm]{Claim}

\newtheorem{ntn}[thm]{Notation}
\newtheorem*{ntn-plain}{Notation}
\newtheorem{setup}[thm]{Setup}
\newtheorem{rem}[thm]{Remark}

\newtheorem*{rem-plain}{Remark}

\newcommand{\inv}{^{-1}}
\newcommand{\from}{\colon}

\newcommand{\imp}{\Rightarrow}
\newcommand{\lto}{\longrightarrow}

\newcommand{\x}{\times}

\newcommand{\surj}{\twoheadrightarrow}

\newcommand{\isom}{\cong}
\newcommand{\defn}{\coloneqq}
\newcommand{\ndef}{\eqqcolon}
\newcommand{\tensor}{\otimes}
\newcommand{\id}{\mathrm{id}}

\newcommand{\wt}{\widetilde}

\renewcommand{\d}{\mathrm d}

\newcommand{\delbar}{\overline\partial}
\newcommand{\ddc}{\mathrm d \mathrm d^c}

\newcommand{\ddual}{^{\smash{\scalebox{.7}[1.4]{\rotatebox{90}{\textup\guilsinglleft} \hspace{-.5em} \rotatebox{90}{\textup\guilsinglleft}}}}}
\newcommand{\acts}{\ \rotatebox[origin=c]{-90}{\ensuremath{\circlearrowleft}}\ }

\newcommand{\factor}[2]{\left. \raise 2pt\hbox{$#1$} \right/\hskip -2pt \raise -2pt\hbox{$#2$}}

\newdir{ ir}{{}*!/-5pt/@^{(}} 
\newdir{ il}{{}*!/-5pt/@_{(}} 

\DeclareMathOperator{\rk}{rk}
\DeclareMathOperator{\Sym}{Sym}
\DeclareMathOperator{\Aut}{Aut}

\newcommand{\lref}{\labelcref}

\newcommand{\set}[1]{\left\{ #1 \right\}}

\def\rd#1.{\lfloor{#1}\rfloor}
\def\rp#1.{\lceil{#1}\rceil}
\def\tw#1.{\langle{#1}\rangle}

\newcommand{\la}{\langle}
\newcommand{\ra}{\rangle}

\renewcommand{\O}[1]{\mathscr{O}_{#1}}
\newcommand{\Omegap}[2]{\Omega_{#1}^{#2}}

\newcommand{\Omegar}[2]{\Omega_{#1}^{[#2]}}
\newcommand{\T}[1]{\mathcal{T}_{#1}}

\newcommand{\canmod}{\mathrm{can}}

\newcommand{\reg}[1]{{#1}_{\mathrm{reg}}}

\newcommand{\cc}[2]{\mathrm{c}_{#1} \!\left( #2 \right)}

\def\Hnought#1.#2.{\mathit{\Gamma} \!\left( #1, #2 \right)}
\def\HH#1.#2.#3.{\mathrm{H}^{#1} \!\left( #2, #3 \right)}
\def\hh#1.#2.#3.{h^{#1} \!\left( #2, #3 \right)}
\def\RR#1.#2.#3.{R^{#1} #2_* #3}
\def\HHc#1.#2.#3.{\mathrm{H}_{\mathrm{c}}^{#1} \!\left( #2, #3 \right)}
\def\Hh#1.#2.#3.{\mathrm{H}_{#1} \!\left( #2, #3 \right)}
\def\Hom#1.#2.{\mathrm{Hom} \!\left( #1, #2 \right)}
\def\End#1.{\mathrm{End} \!\left( #1 \right)}
\def\sHom#1.#2.{\mathscr{H}\!om \!\left( #1, #2 \right)}
\def\Ext#1.#2.#3.{\mathrm{Ext}^{#1} \!\left( #2, #3 \right)}
\def\sExt#1.#2.#3.{\mathscr{E}\!xt^{#1} \!\left( #2, #3 \right)}
\def\Link#1.#2.{\mathrm{Link} \!\left( #1, #2 \right)}

\renewcommand{\H}[1]{\mathbb H^{#1}}
\renewcommand{\P}[1]{\mathbb P^{#1}}

\newcommand{\SL}[2]{\mathrm{SL}(#1, #2)}

\newcommand{\U}[1]{\mathrm{U}(#1)}
\newcommand{\SU}[1]{\mathrm{SU}(#1)}
\newcommand{\ke}{\mathrm{KE}}

\newcommand{\kahler}{K{\"{a}}hler\xspace}

\newcommand{\qe}{quasi-\'etale\xspace}

\DeclareMathOperator{\Exc}{Exc}

\DeclareMathOperator{\Ric}{Ric}
\DeclareMathOperator{\tr}{tr}

\newcommand{\eps}{\varepsilon}
\renewcommand{\theta}{\vartheta}
\renewcommand{\phi}{\varphi}

\newcommand{\Q}{\ensuremath{\mathbb Q}}
\newcommand{\R}{\ensuremath{\mathbb R}}
\newcommand{\C}{\ensuremath{\mathbb C}}
\newcommand{\bB}{\ensuremath{\mathbb B}}
\newcommand{\bH}{\ensuremath{\mathbb H}}

\newcommand{\bP}{\ensuremath{\mathbb P}}

 \newcommand{\frb}{\mathfrak b} \newcommand{\frc}{\mathfrak c}

\newcommand{\frA}{\mathfrak A} \newcommand{\frS}{\mathfrak S}

  \newcommand{\sC}{\mathscr C}
 \newcommand{\sE}{\mathscr E} \newcommand{\sF}{\mathscr F}
\newcommand{\sG}{\mathscr G}  
  \newcommand{\sL}{\mathscr L}

\newcommand{\cD}{\mathcal D}

\definecolor{forrest}{RGB}{81,133,49}
\definecolor{mydarkblue}{RGB}{10,92,153}

\title{Semispecial tensors and quotients of the polydisc}

\author{Patrick Graf}
\address{Lehrstuhl f\"ur Mathematik I, Universit\"at Bayreuth, 95440 Bayreuth, Germany}
\email{\href{mailto:patrick.graf@uni-bayreuth.de}{patrick.graf@uni-bayreuth.de}}
\urladdr{\href{https://patrickgraf.gitlab.io/en/}{www.graficland.uni-bayreuth.de}}

\author{Aryaman Patel}
\address{AG Lazi\'c, Universit\"at des Saarlandes, 66123 Saarbr\"ucken, Germany}
\email{\href{mailto:aryaman.patel@math.uni-sb.de}{aryaman.patel@math.uni-sb.de}}
\urladdr{\href{https://sites.google.com/view/aryamanpatel/home}{https://sites.google.com/view/aryamanpatel/}}

\date{May 6, 2026}

\hypersetup{
  pdfauthor={Patrick Graf and Aryaman Patel},
  pdftitle={Semispecial tensors and quotients of the polydisc},
  pdfkeywords={Varieties with ample canonical divisor, uniformization, klt singularities, Kähler--Einstein metrics, holonomy groups, polydiscs, Bochner principle},
  pdfstartview={Fit},
  pdfpagelayout={OneColumn},
  pdfpagemode={UseNone},
  linktoc=all,
  breaklinks,
  linkcolor=[RGB]{0 0 96},
  citecolor=[RGB]{96 0 0},
  urlcolor=[RGB]{0 96 0},
  colorlinks}

\begin{document}

\begin{abstract}
Let $X$ be a complex-projective variety with klt singularities and ample canonical divisor.
We prove that $X$ is a quotient of the polydisc by a group acting properly discontinuously and freely in codimension one if and only if $X$ admits a semispecial tensor with reduced hypersurface.
This extends a result of Catanese and Di Scala to singular spaces, and answers a question raised by these authors.
As a key step in the proof, we establish the Bochner principle for holomorphic tensors on klt spaces in the negative K{\"{a}}hler--Einstein case.
\end{abstract}

\maketitle

\section{Introduction}

The uniformization theorem of Koebe and Poincar\'e (1907) states that the only simply connected Riemann surfaces are the Riemann sphere $\P1_\C$, the complex plane $\C$, and the unit disc $\bB^1$.
In higher dimensions, one quickly realizes that a similar result is impossible, as there are just too many simply connected complex manifolds in any given dimension $n \ge 2$.
One standard way to solve this issue is to modify the question by fixing some simply connected ``model space'' $\cD$ with well-understood geometry and asking for a geometric or numerical characterization of those compact complex manifolds (or smooth projective varieties) whose universal cover $\wt X$ is biholomorphic to $\cD$.

The first such result is the well-known Miyaoka--Yau inequality~\cite{Yau78} for smooth projective varieties $X$ with ample canonical divisor $K_X$, together with the statement that equality is attained if and only if $X$ is uniformized by the unit ball $\cD = \bB^n$.
Using the theory of Higgs bundles, Simpson~\cite{Simpson88} gave a more general uniformization result for arbitrary bounded symmetric domains $\cD = \factor{G_0}{K_0}$.
Roughly speaking, he proved that $\wt X \isom \cD$ if and only if
\begin{itemize}[leftmargin=2em]
\item[\textsf{(A)}] the tangent bundle $\T X$ has a reduction of structure group to $K$, the complexification of $K_0$, and
\item[\textsf{(B)}] a certain vector bundle $\sE$ associated to this reduction satisfies the Chern class equality $\cc2\sE \cdot [K_X]^{n-2} = 0$.
\end{itemize}
We would like to highlight two extreme cases of this equivalence:
\begin{itemize}
\item If $\cD = \bB^n = \factor{\SU{1, n}}{\mathrm S(\U1 \x \U n)}$ is the \emph{unit ball}, then condition~\textsf{(A)} becomes vacuous, while \textsf{(B)} boils down to having equality in the Miyaoka--Yau inequality.
One thus recovers the aforementioned result.
\item If $\cD = \H n = \factor{{\SL2\R}^n}{{\U1}^n}$ is the \emph{polydisc}, then \textsf{(A)} boils down to $\T X$ splitting into a direct sum of line bundles, while \textsf{(B)} becomes vacuous (as observed by Beauville~\cite{Beauville00}).
\end{itemize}

In a slightly different vein, Catanese and Di Scala~\cite{CataneseDiScala13, CataneseDiScala14} have given several uniformization results for various types of bounded symmetric domains $\cD$ in terms of the existence of certain kinds of holomorphic tensors on $X$.
For example, in the case of the polydisc, they proved that $\wt X \isom \H n$ if and only if $X$ admits a so-called \emph{semispecial tensor} with reduced hypersurface, i.e.~a nonzero section
\[ \psi \in \HH0.X.\Sym^{n}(\Omegap X1)(-K_X) \tensor \eta. \setminus \set0 \]
such that for any $p \in X$, the scheme-theoretic hypersurface $\set{ \psi_p = 0 } \subset \bP(T_p X)$ is reduced, where $\eta$ is a line bundle on $X$ such that $\eta^{\tensor 2} \isom \O X$~\cite[Thm.~1.1]{CataneseDiScala13}.

More recently, Catanese~\cite{Catanese25} has proved similar results for Deligne--Mostow orbifolds uniformized by bounded symmetric domains $\cD$ of tube type.
Roughly speaking, a Deligne--Mostow orbifold is a pair $(Z, \Delta)$ which is covered by local orbifold charts $\phi \from \Omega \to U = \factor \Omega G$ with $\Omega \subset \C^n$ an open subset (see~\cite[Def.~3(4)]{Catanese25} for the precise definition).
Catanese also asks whether the Deligne--Mostow condition can be replaced by the weaker assumption of having klt singularities.
Our main result below will answer this question affirmatively, at least in the special case where $\cD$ is the polydisc and the orbifold divisor $\Delta$ is empty.

\subsection*{Main result}

The above results apply to projective manifolds with ample canonical divisor $K_X$.
Hence it seems natural to look for a way of producing such manifolds.
The Minimal Model Program provides such a way: by the landmark results of~\cite{BCHM10}, any projective manifold $X$ of general type (i.e.~with maximal Kodaira dimension) admits a canonical model $X_\canmod$, which is birational to $X$ and has ample canonical divisor $K_{X_\canmod}$.
But there is a catch: in general, $X_\canmod$ will not be a manifold any longer, but a variety with \emph{canonical singularities}.
Consequently, the above-mentioned uniformization results cannot be applied to canonical models in general.

Therefore, in recent years there has been a lot of activity centered around the uniformization problem for varieties with canonical, or more generally klt, singularities:~\cite{GKPT1_ENS}, \cite{Patel23}, \cite{MYeq} and~\cite{GrafPatel24}, to name a few.
However, all of these focus on statements in the style of Miyaoka--Yau and Simpson.
In this article, we investigate what can be said about the Catanese--Di Scala approach in the singular setting.
We will concentrate on the case of the polydisc $\cD = \H n$ and leave the other cases to a future work.

Note that the uniformization problem in the smooth case can be rephrased as follows: \emph{given a projective manifold $X$ with ample canonical divisor and a bounded symmetric domain $\cD$, when is $X \isom \factor\cD\Gamma$ for some discrete cocompact subgroup $\Gamma \subset \Aut(\cD)$ acting freely?}
It is this formulation (and not the one involving the universal cover) that lends itself best to generalization to the singular case.
Namely, one may simply replace the freeness requirement on $\Gamma$ by freeness \emph{in codimension one}.
Accordingly, our main result is as follows.

\stepcounter{thm}

\begin{bigthm}[Characterization of singular polydisc quotients, \cref{main1'}] \label{main1}
Let $X$ be an $n$-dimensional normal projective variety with klt singularities and ample canonical divisor $K_X$.
The following are equivalent:
\begin{enumerate}
\item\label{main1.1} We have $X \isom \factor{\H n}\Gamma$, where $\Gamma \subset \Aut(\H n)$ is a discrete cocompact subgroup whose action is free in codimension one.
\item\label{main1.3} There is a semispecial tensor $\psi$ on $X$ such that for some point $p \in \reg X$, the scheme-theoretic hypersurface $\set{ \psi_p = 0 } \subset \bP(T_p X)$ is reduced.
\end{enumerate}
\end{bigthm}

\begin{rem-plain}
For the notion of semispecial tensors on klt spaces, see \cref{def semispecial}.
It is a straightforward extension to normal varieties of the corresponding notion introduced by Catanese--Di Scala in~\cite{CataneseDiScala13}, as mentioned on the previous page.
\end{rem-plain}

\subsection*{The Bochner principle}

The well-known Bochner principle states that on a compact \kahler--Einstein manifold $X$, any holomorphic tensor of contravariant degree~$p$ and covariant degree~$q$ is parallel, or even zero (for certain values of $p$ and $q$ depending on the sign of the Ricci tensor).
The original reference for this result is~\cite{BochnerYano53}, but~\cite[Thm.~1(2)]{Kobayashi80} is closer to the version that we present below.

The Bochner principle is one of the main ingredients to the proofs in~\cite{CataneseDiScala13, CataneseDiScala14}, since it provides a link between the holonomy of $X$ and the existence of semispecial and slope zero tensors.
It is therefore not surprising that establishing a singular version of this principle is also one of the main steps in our argument.

\stepcounter{thm}

\begin{bigthm}[Bochner principle, \cref{bochner}] \label{bochner main}
Let $X$ be a normal projective variety with klt singularities and ample canonical divisor $K_X$.
Let $\omega_\ke$ be the singular \kahler--Einstein metric on $X$ constructed by~\cite{EGZ09} and~\cite{BermanGuenancia14}.
Consider a holomorphic tensor
\[ \tau \in \HH0.\reg X.\T{\reg X}^{\tensor p} \tensor \Omega_{\reg X}^{\tensor q}.. \]
\begin{enumerate}
\item If $p = q$, then $\tau$ is parallel with respect to $\omega_\ke$.
\item If $p > q$, then $\tau = 0$.
\end{enumerate}
\end{bigthm}

\begin{rem-plain}
The $p > q$ case of \cref{bochner main} is a direct consequence of the known semi\-stability of $\T X$ with respect to $K_X$~\cite[Thm.~A]{Guenancia16}.
However in our applications we mostly need the $p = q$ case.
\end{rem-plain}

There are two versions of the Bochner principle for klt spaces in the literature, both of which deal with the case where $K_X$ is numerically trivial:~\cite[Thm.~A]{GGK19} and~\cite[Thm.~A]{BochnerCGGN}.
In the first paper, the parallelity statement is first established for \emph{subbundles} of a given tensor bundle, then the corresponding statement for \emph{sections} is deduced.
This relies on two important facts about the holonomy group of $(\reg X, \omega_\ke)$ in the Ricci-flat case: that it is compact, and that its identity component is semisimple.
In our setting, the first of these is not known, while the second one is actually false: for example, the holonomy of the unit ball $\bB^n$ is $\U n$ and that of the polydisc is ${\U1}^n$, neither of which is semisimple.
Therefore, our proof of \cref{bochner main} follows the general strategy of~\cite{BochnerCGGN} instead, with extra steps due to the more complicated form of the Ricci tensor in our case.
We believe that the Bochner principle for subbundles could be established using the techniques from~\cite{GGK19}, however we do not need this statement.

\subsection*{Applications}

In~\cite[p.~423]{CataneseDiScala13}, Catanese and Di Scala asked whether it is possible to remove the assumption that $K_X$ is ample from their uniformization result, replacing it by the condition that $X$ be of general type.
Using \cref{main1}, we show that this is indeed the case.

\begin{cor}[Uniformization of canonical models] \label{cor gen type}
Let $X$ be an $n$-dimensional projective manifold of general type.
Assume that $X$ admits a semispecial tensor~$\psi$ such that for some point $p \in X$, we have $\psi_p \ne 0$ as a polynomial on $T_p X$ and the scheme-theoretic hypersurface $\set{ \psi_p = 0 } \subset \bP(T_p X)$ is reduced.
Then the canonical model $X_\canmod$ of $X$ is a polydisc quotient $\factor{\H n}\Gamma$ in the sense of~\lref{main1.1}.
\end{cor}

\begin{rem-plain}
In the above situation, it would be natural to make a statement about the universal cover $\wt X$ of $X$ itself, such as ``$\wt X$ is bimeromorphic to the polydisc''.
We only have quite weak statements in this direction, cf.~\cref{rem unif}.
\end{rem-plain}

In~\cite{Kazhdan71}, Kazhdan proved that if $X$ is an arithmetic variety (the quotient of a Hermitian symmetric space by a torsion-free arithmetic group), then for any field automorphism $\sigma \in \Aut(\C / \Q)$, the conjugate variety $X^\sigma$ is again arithmetic.
Simpson observed that his uniformization theorem yields corollaries in the style of Kazhdan's result~\cite[Cor.~9.5]{Simpson88}.
In our situation, we have the following.

\begin{cor}[Conjugates of polydisc quotients] \label{cor conjugate}
Let $X$ be a projective variety that is a quotient of the polydisc in the sense of~\lref{main1.1}.
Then for any automorphism $\sigma \in \Aut(\C / \Q)$, the conjugate variety $X^\sigma$ is again a polydisc quotient in that sense.
\end{cor}

Finally, we remark that in low dimensions, the existence of a semispecial tensor (without the reducedness hypothesis) is sufficient to characterize quotients of the polydisc.
This is based on an argument of Catanese and Franciosi~\cite[Thm.~1.9]{CataneseFranciosi09}.

\begin{cor}[Three-dimensional polydisc quotients] \label{cor dim 3}
Let $X$ be a normal projective variety with klt singularities and ample canonical divisor $K_X$, of dimension $n \le 3$.
Then, the following are equivalent:
\begin{enumerate}
\item\label{dim3.1} We have $X \isom \factor{\H n}\Gamma$, where $\Gamma \subset \Aut(\H n)$ is a discrete cocompact subgroup whose action is free in codimension one.
\item\label{dim3.2} There is a semispecial tensor $\psi$ on $X$.
\end{enumerate}
\end{cor}

\subsection*{Acknowledgements}

We would like to thank Henri Guenancia for helpful discussions.
The first author was funded by the Deutsche Forschungsgemeinschaft (DFG, German Research Foundation) -- Projektnummer 521356266.
The second author was supported by the Deutsche Forschungsgemeinschaft (DFG, German Research Foundation) -- Project ID 286237555 (TRR 195) and Project ID 530132094.

\section{Notation and conventions}

We work over the field of complex numbers \C.
For the reader's convenience, we repeat the definition of klt singularities.

\begin{dfn}[Klt singularities, cf.~\protect{\cite[Cor.~2.32 and Def.~2.34]{KM98}}]
Let $(X, \Delta)$ be a pair, where $X$ is a normal variety and $\Delta = \sum_i a_i \Delta_i$ is a divisor on $X$, where the $\Delta_i$ are distinct prime divisors and $a_i \in \Q \cap [0, 1)$ for all $i$.
Assume moreover that $K_X + \Delta$ is \Q-Cartier.
Fix a log resolution $f \from Y \to X$ of the pair $(X, \Delta)$ and write
\[ K_Y + f\inv_* \Delta \; \sim_\Q \; f^*( K_X + \Delta ) + \sum_i a(E_i, X, \Delta) E_i \]
where the $E_i$ are the irreducible exceptional divisors of $f$.
The pair $(X, \Delta)$ is said to have \emph{Kawamata log terminal (klt) singularities} if
\[ a(E_i, X, \Delta) > -1 \quad \text{for all $i$.} \]
We say that $X$ has klt singularities if the pair $(X, \emptyset)$ has klt singularities.
\end{dfn}

\begin{ntn} \label{refl tensor}
The \emph{reflexive tensor product} of two coherent sheaves $\sF$ and $\sG$ on a normal variety $X$ is
\[ \sF [\tensor] \sG \defn (\sF \tensor \sG)\ddual, \]
where $(-)\ddual$ denotes the double dual.
If $\sG = \O X(D)$ is a Weil divisorial sheaf (i.e.~a reflexive coherent sheaf of rank one), we set
\[ \sF(D) \defn \sF [\tensor] \O X(D). \]
If $f \from Y \to X$ is a morphism between normal varieties, the \emph{reflexive pullback} of $\sF$ is
\[ f^{[*]} \sF \defn (f^* \sF)\ddual. \]
We will use this notation also when pulling back local or global sections of $\sF$.
\end{ntn}

For the reader's convenience, we include a notion from linear algebra that will be used in the proof of the Bochner principle.

\begin{dfn}
Let $V$ be a finite-dimensional Hermitian vector space.
A self-adjoint endomorphism $f \from V \to V$ is called \emph{positive} if $\la f(v), v \ra > 0$ for all $v \ne 0$.
It is called \emph{semipositive} if $\la f(v), v \ra \ge 0$ for all $v \in V$.
If $f$ is (semi)positive, we write $f > 0$ ($f \ge 0$).
\end{dfn}

\begin{lem} \label{lem semipos}
If $f \from V \to V$ is semipositive, then $f \le \tr(f) \cdot \id_V$.
\end{lem}

\begin{proof}
By the spectral theorem, $f$ is diagonalizable.
After choosing a basis of eigenvectors, we may assume that $f$ is given by a diagonal matrix.
Since the eigenvalues of $f$ are real and non-negative, the claim then becomes obvious.
\end{proof}

\section{The Bochner principle for klt spaces}

In this section, we formulate a more precise version of \cref{bochner main} and prove the relevant corollary about semispecial tensors.
Throughout, $X$~will denote an $n$-dimensional normal projective variety with klt singularities and ample canonical divisor $K_X$.

As already mentioned in the introduction, the following notion goes back to~\cite{CataneseDiScala13} if $X$ is smooth.
Our definition is the natural extension of theirs to normal varieties.

\begin{dfn} \label{def semispecial}
Let $X$ be as above.
\begin{enumerate}
\item A \emph{slope zero tensor} is a nonzero section
\[ 0 \ne \psi \in \HH0.X.\Sym^{mn}(\Omegap X1)(-mK_X). \]
for some positive integer $m > 0$.
\item A \emph{semispecial tensor} is a nonzero section
\[ 0 \ne \psi \in \HH0.X.\Sym^{n}(\Omegap X1)(-K_X) [\tensor] \eta. \]
for some rank one reflexive sheaf $\eta$ on $X$ such that $\eta^{[2]} \isom \O X$.
\end{enumerate}
\end{dfn}

\begin{rem-plain}
Note that the sheaves appearing above are by definition reflexive, cf.~\cref{refl tensor}.
Also, a Chern class calculation shows that $\cc1{\Sym^{mn}(\Omegap X1)(-mK_X)} = 0$, which may help to explain the terminology.
\end{rem-plain}

\begin{setup} \label{bochner setup}
Let $X$ be as above, and let $p, q \ge 0$ be non-negative integers.
Consider the reflexive sheaf
\[ \sE \defn \T X^{\tensor p} [\tensor] \Omegap X{\tensor q}. \]
According to~\cite[Theorem~7.8]{EGZ09} and~\cite[Thm.~A, Cor.~5.7]{BermanGuenancia14}, there exists a unique closed, positive current $\omega_\ke \in [K_X] \in \HH2.X.\R.$ with bounded potentials, smooth on $\reg X$ and satisfying $\Ric(\omega_\ke) = - \omega_\ke$ on this locus.
The \kahler metric $\omega_\ke$ on $\reg X$ induces a smooth hermitian metric on $\T{\reg X}$ and in turn on $\sE \big|_{\reg X}$, with Chern connection~$D_{\sE}$.
\end{setup}

\begin{thm}[Bochner principle] \label{bochner}
Let $(X, \omega_\ke)$ and $\sE$ be as in \cref{bochner setup}.
Consider a holomorphic tensor
\[ \tau \in \HH0.X.\sE. = \HH0.\reg X.\T{\reg X}^{\tensor p} \tensor \Omega_{\reg X}^{\tensor q}.. \]
\begin{enumerate}
\item\label{bochner.2} If $p = q$, then $\tau$ is parallel with respect to $\omega_\ke$, i.e.~$D_{\sE} \tau = 0$ on $\reg X$.
\item\label{bochner.1} If $p > q$, then $\tau = 0$.
\end{enumerate}
\end{thm}

\begin{cor} \label{bochner cor}
Let $(X, \omega_\ke)$ be as in \cref{bochner setup}.
\begin{enumerate}
\item\label{bochner cor.1} Any slope zero tensor $\psi$ on $X$ is parallel with respect to $\omega_\ke$.
\item\label{bochner cor.2} Any semispecial tensor $\psi$ on $X$ is parallel with respect to $\omega_\ke$ and the flat metric on $\eta\big|_{\reg X}$.
\end{enumerate}
\end{cor}

\begin{rem} \label{flat eta}
The flat metric on $\eta\big|_{\reg X}$ is given as follows:
let $\phi \from X' \to X$ be the \qe double cover such that $\phi^{[*]} \eta \isom \O{X'}$.
Choose a flat (i.e.~constant) metric on this trivial line bundle.
By uniqueness of flat metrics (up to scalars), this metric is $\phi$-invariant and hence descends to $\eta\big|_{\reg X}$.
\end{rem}

\begin{proof}[Proof of \cref{bochner cor}]
As $\Sym^{mn}(\Omegap{\reg X}1) \subset \Omegap{\reg X}{\tensor mn}$ and $\O{\reg X}(-mK_{\reg X}) \subset \T{\reg X}^{\tensor mn}$, statement~\lref{bochner cor.1} follows from \cref{bochner} with $p = q = mn$.

In the semispecial case~\lref{bochner cor.2}, we consider the double cover $\phi \from X' \to X$ as in \cref{flat eta}.
Note that $X'$ still has klt singularities~\cite[Prop.~5.20]{KM98} and that $K_{X'} = \phi^* K_X$ remains ample.
Let $\omega_{\ke, X'}$ be the \kahler--Einstein current on~$X'$ as in \cref{bochner setup}.
By the uniqueness part of~\cite[Theorem~7.8]{EGZ09}, we have $\omega_{\ke, X'} = \phi^* \omega_\ke$.

We can invoke~\lref{bochner cor.1} on $X'$ to see that the slope zero tensor $\phi^{[*]} \psi$ is parallel with respect to $\omega_{\ke, X'} = \phi^* \omega_\ke$.
This implies that $\psi$ itself is parallel with respect to $\omega_\ke$.
\end{proof}

\section{Proof of the Bochner principle}

In this section, we prove \cref{bochner}.
Since the proof is quite lengthy, it is divided into several steps.

\subsection*{Step~0: Setup}

We choose a projective resolution $\pi \from Y \to X$ such that
\[ \Exc(\pi) \ndef F = \sum_{j=1}^\ell F_j \]
is an snc divisor, and we set $Y^\circ \defn Y \setminus F$.
For each $j$, we fix a section $s_j \in \HH0.Y.\O Y(F_j).$ whose divisor is precisely $F_j$.
We also use the discrepancy formula
\[ K_Y = \pi^* K_X + \sum_{j=1}^\ell a_j F_j, \]
where the $a_j \in (-1, \infty) \cap \Q$ because $X$ is klt.
Furthermore, we choose an ample line bundle $A$ on $Y$ and a \kahler form $\omega_A$ on $Y$ representing $\cc1A$.
Finally, we will consider the vector bundle
\[ \sE_Y \defn \T Y^{\tensor p} \tensor \Omegap Y{\tensor q} \]
on $Y$, where $p$ and $q$ have the same value as in the definition of $\sE$.
By~\cite[Ch.~II, Lemma~5.3]{Har77}, there is an integer $k \gg 0$ such that $\tau$ extends to a holomorphic section $\sigma \in \HH0.Y.\sE_Y(kF).$.
(Actually, $k$ can be chosen independent of $\tau$, but we do not need this fact.)

\subsection*{Step~1: Smooth metrics on exceptional line bundles}

For any $1 \le j \le \ell$, choose a smooth hermitian metric $h_j$ on $\O Y(F_j)$, also denoted $\la -, - \ra_j$ or $| - |_j$.
Let $D_j = D_j' + \delbar$ be the Chern connection on $(\O Y(F_j), h_j)$ and let $\theta_j \defn i \, \Theta_{h_j}(F_j)$ be the corresponding curvature form.

We define for any $\eps > 0$ another smooth metric $h_{j, \eps}$ on $\O Y(F_j)$ by setting
\[ h_{j, \eps} \defn \frac1{|s_j|_j^2 + \eps^2} \cdot h_j. \]
The curvature form $\theta_{j, \eps} \defn i \, \Theta_{h_{j, \eps}}(F_j)$ of this new metric is given by
\[ \theta_{j, \eps} = \theta_j + \ddc \log \big( |s_j|_j^2 + \eps^2 \big). \]
As in~\cite[Rem.~9.3]{GGK19}, this can also be written
\[ \theta_{j, \eps} = \underbrace{\frac{\eps^2}{ \big( |s_j|_j^2 + \eps^2 \big)^2} \cdot \la D_j' s_j, D_j' s_j \ra_j}_{\ndef \beta_{j, \eps}}
\;\; + \;\;
\underbrace{\frac{\eps^2}{|s_j|_j^2 + \eps^2} \cdot \theta_j}_{\ndef \gamma_{j, \eps}}. \]
Furthermore, we set 
\[ h_F \defn \prod_{j=1}^\ell h_j \qquad \text{and} \qquad
h_{F, \eps} \defn \prod_{j=1}^\ell h_{j, \eps} = \frac{1}{\prod_{j=1}^\ell \big( |s_j|_j^2 + \eps^2 \big)} \cdot h_F. \]
These are smooth metrics on the line bundle $\O Y(F)$.
Finally, by formally setting $\eps = 0$ in the above definitions, we can define singular metrics $h_{j, 0}$ on $\O Y(F_j)$ and $h_{F, 0}$ on $\O Y(F)$.
These metrics are smooth when restricted to $Y^\circ$.

\subsection*{Step~2: Approximate \kahler--Einstein metrics}

For any two positive real numbers $t, \eps > 0$, we consider the unique \kahler metric $\omega_{t, \eps}$ on $Y$ which satisfies $\omega_{t, \eps} \in \cc1{\pi^* K_X + tA}$ and
\begin{equation} \label{approx ke}
\Ric \omega_{t, \eps} = - \omega_{t, \eps} + t \omega_A - \sum_{j=1}^\ell a_j \theta_{j, \eps}.
\end{equation}
For existence and uniqueness of $\omega_{t, \eps}$, see~\cite{Kobayashi84} and~\cite{TianYau87}.
It follows from~\cite[Thm.~4.5]{BermanGuenancia14} that as $t, \eps \to 0$, the forms $\omega_{t, \eps}$ converge to $\pi^* \omega_\ke$ on $Y^\circ$ in the $\sC^\infty_{\mathrm{loc}}$-topology.

\subsection*{Step~3: Curvature estimates}

The \kahler metric $\omega_{t, \eps}$ induces hermitian metrics on $\T Y^{\tensor p}$, $\Omegap Y{\tensor q}$ and $\sE_Y$, all of which we will call $h_{t, \eps}$.
What is more, the metric $h_{t, \eps} \tensor h_{F, \eps}^{\tensor k}$ on $\sE_Y(kF)$ will also be denoted by $h_{t, \eps}$.
To lighten notation, we will often drop the dependence on $t$ and $\eps$.
That is, we will denote the Chern connection $D_{t, \eps}$ on $(\sE_Y(kF), h_{t, \eps})$ simply by $D$ and set $| - | \defn | - |_{h_{t, \eps}}$ as well as $\la -, - \ra \defn \la -, - \ra_{h_{t, \eps}}$.

Consider the section $\sigma \in \HH0.Y.\sE_Y(kF).$ from Step~0.
Then according to the computations of~\cite[p.~256]{BochnerCGGN}, we have the estimate
\begin{equation} \label{main est}
0 \le \int_Y \frac{ |D \sigma|^2 }{ \big( |\sigma|^2 + 1 \big)^2 } \wedge \omega_{t, \eps}^{n-1} \le
\int_Y \frac{ \big\la i \, \Theta(\sE_Y(kF), h_{t, \eps}) \sigma, \, \sigma \big\ra }{ |\sigma|^2 + 1 } \wedge \omega_{t, \eps}^{n-1},
\end{equation}
and the goal is to understand the right-hand side integral.
But first, let us recall some auxiliary notation from~\cite{GGK19} and~\cite{BochnerCGGN}.

\begin{ntn}[\protect{\cite[Not.~9.11]{GGK19}}]
If $f$ is an endomorphism of some finite-dimensional complex vector space $V$, then $f^{\boxtimes p}$ is the endomorphism of $V^{\tensor p}$ defined by
\[ f^{\boxtimes p} \defn \sum_{i=1}^p \id_V^{\tensor(i-1)} \tensor f \tensor \id_V^{\tensor(p-i)}. \]
\end{ntn}

\noindent
Note that this is a linear operation: $(f + \lambda g)^{\boxtimes p} = f^{\boxtimes p} + \lambda \cdot g^{\boxtimes p}$ for $\lambda \in \C$.

\begin{ntn}[\protect{\cite[Not.~9.13]{GGK19}}]
If $\alpha$ is a smooth $(1, 1)$-form on $Y$, we denote by $\sharp_{t, \eps} \alpha$ the smooth endomorphism of $\T Y$ given by ``raising the indices'' in the $(0, 1)$~part of $\alpha$ with respect to the \kahler metric $\omega_{t, \eps}$.
As before, we will simply write $\sharp \alpha$ without reference to $t$ and $\eps$.
\end{ntn}

\begin{ntn}[\protect{\cite[Expl.~9.8]{GGK19}}]
The symbol $\tr_{t, \eps}$ will denote the trace relative to the \kahler metric~$\omega_{t, \eps}$.
That is, if $\alpha$ is an $F$-valued $(1, 1)$-form on $Y$, for some vector bundle $F$, then $\tr_{t, \eps} \alpha$ is the unique section of $F$ such that
\[ (\tr_{t, \eps} \alpha) \tensor \omega_{t, \eps}^n = n \cdot \alpha \wedge \omega_{t, \eps}^{n-1}. \]
\end{ntn}

We have (cf.~\cite[Claim~9.14]{GGK19})
\begin{align*}
\tr_{t, \eps} \!\big( i \, \Theta(\T Y^{\tensor p}, h_{t, \eps}) \big) & = \phantom - (\sharp \Ric \omega_{t, \eps})^{\boxtimes p} \quad \text{and} \\
\tr_{t, \eps} \!\big( i \, \Theta(\Omegap Y{\tensor q}, h_{t, \eps}) \big) & = - \overline{(\sharp \Ric \omega_{t, \eps})}^{\boxtimes q}.
\end{align*}
By~\lref{approx ke} and since $\sharp \omega_{t, \eps} = \id_{\T Y}$, it holds that
\begin{align*}
(\sharp \Ric \omega_{t, \eps})^{\boxtimes p} & = - p \cdot \id_{\T Y^{\tensor p}} + t \cdot (\sharp \omega_A)^{\boxtimes p} - \sum_{j=1}^\ell a_j (\sharp \theta_{j, \eps})^{\boxtimes p}, \\
- \overline{(\sharp \Ric \omega_{t, \eps})}^{\boxtimes q} & = \phantom - q \cdot \id_{\Omegap Y{\tensor q}} - t \cdot \overline{(\sharp \omega_A)}^{\boxtimes q} + \sum_{j=1}^\ell a_j \overline{(\sharp \theta_{j, \eps})}^{\boxtimes q}
\end{align*}
and so in summary
\begin{align*}
\tr_{t, \eps} \!\big( i \, \Theta(\sE_Y, h_{t, \eps}) \big) & = (\sharp \Ric \omega_{t, \eps})^{\boxtimes p} \tensor \id_{\Omegap Y{\tensor q}} - \id_{\T Y^{\tensor p}} \tensor \overline{(\sharp \Ric \omega_{t, \eps})}^{\boxtimes q} \\
& = (q - p) \id_{\sE_Y} + t \cdot \underbrace{\left( (\sharp \omega_A)^{\boxtimes p} \tensor \id_{\Omegap Y{\tensor q}} - \id_{\T Y^{\tensor p}} \tensor \overline{(\sharp \omega_A)}^{\boxtimes q} \right)}_{\ndef \frb} \\
& \qquad - \underbrace{\sum_{j=1}^\ell a_j \left( (\sharp \theta_{j, \eps})^{\boxtimes p} \tensor \id_{\Omegap Y{\tensor q}} - \id_{\T Y^{\tensor p}} \tensor \overline{(\sharp \theta_{j, \eps})}^{\boxtimes q} \right)}_{\ndef \frc}.
\end{align*}
Note that $\frb$ and $\frc$ are endomorphisms of $\sE_Y$.
Coming back to~\lref{main est}, the standard formula for the curvature of a tensor product reads in our case
\[ i \, \Theta(\sE_Y(kF), h_{t, \eps}) = i \, \Theta(\sE_Y, h_{t, \eps}) \tensor \id_{\O Y(kF)} + \id_{\sE_Y} \tensor i \, \Theta \big( \O Y(kF), h_{F, \eps}^{\tensor k} \big). \]
Choosing a local generator $e$ of the line bundle $\O Y(kF)$ and locally writing $\sigma = u \tensor e$ with $u$ a local section of $\sE_Y$, we get
\begin{equation}
\big\la i \, \Theta(\sE_Y(kF), h_{t, \eps}) \sigma, \, \sigma \big\ra =
\underbrace{\big\la i \, \Theta(\sE_Y, h_{t, \eps}) u, \, u \big\ra \cdot |e|_{h_{F, \eps}^{\tensor k}}^2}_{\ndef \mathrm{(I)}} +
\underbrace{k \sum_{j=1}^\ell \theta_{j, \eps} \cdot |\sigma|^2}_{\ndef \mathrm{(II)}}.
\end{equation}

\subsection*{Step~4: Computation of term~(I)}

By the definition of $\tr_{t, \eps}$ and the above computations, we have
\begin{align*}
n \cdot \big\la i \, \Theta(\sE_Y, h_{t, \eps}) u, \, u \big\ra \wedge \omega_{t, \eps}^{n-1}
& = \left\la \tr_{t, \eps} \!\big( i \, \Theta(\sE_Y, h_{t, \eps}) \big)(u), u \right\ra \cdot \omega_{t, \eps}^n \\
& = \left[ (q - p) |u|^2 + t \cdot \la \frb(u), u \ra - \la \frc(u), u \ra \right] \cdot \omega_{t, \eps}^n.
\end{align*}
The $\la \frc(u), u \ra$ term can be handled in exactly the same way as in~\cite[p.~257]{BochnerCGGN}, where we use $\omega_A$ as the reference \kahler metric $\omega_Y$ on $Y$.
As for the $\la \frb(u), u \ra$ term, note that $(\sharp \omega_A)^{\boxtimes p}$ is a positive endomorphism because $\omega_A$ is \kahler and that $\tr_{\mathrm{End}} (\sharp \omega_A)^{\boxtimes p} = p n^{p-1} \tr_{\mathrm{End}} \sharp \omega_A = p n^{p-1} \tr_{t, \eps} \omega_A$.
Hence by \cref{lem semipos}
\[ (\sharp \omega_A)^{\boxtimes p} \tensor \id_{\Omegap Y{\tensor q}} \le p n^{p-1} \tr_{t, \eps} \omega_A \cdot \id_{\sE_Y}. \]
The second term of $\frb$ can be dealt with in much the same way, so we obtain
\[ \frb \le C \cdot \tr_{t, \eps} \omega_A \cdot \id_{\sE_Y} \]
for some constant $C = C(n, p, q) > 0$.
Therefore
\begin{align*}
\la \frb(u), u \ra \cdot \omega_{t, \eps}^n & \le C \cdot |u|^2 \cdot \tr_{t, \eps} \omega_A \cdot \omega_{t, \eps}^n \\
& = n \cdot C \cdot |u|^2 \cdot \omega_A \wedge \omega_{t, \eps}^{n-1}.
\end{align*}
Observing (at least for terms $\mathrm{(IV)}$ and $\mathrm{(V)}$) that $\frac{|u|^2 \cdot |e|^2}{|\sigma|^2 + 1} = \frac{|\sigma|^2}{|\sigma|^2 + 1} \le 1$, we arrive at
\begin{align*}
\frac{ \mathrm{(I)} }{ |\sigma|^2 + 1 } \wedge \omega_{t, \eps}^{n-1} & \le \underbrace{\frac{q - p}n \cdot \frac{|\sigma|^2}{|\sigma|^2 + 1} \omega_{t, \eps}^n}_{\ndef \mathrm{(III)}} + \; t \cdot \underbrace{C \cdot \omega_A \wedge \omega_{t, \eps}^{n-1}}_{\ndef \mathrm{(IV)}} \\
& \qquad + \underbrace{C' \cdot \left( \beta_{j, \eps} + \frac{\eps^2}{|s_j|^2 + \eps^2} \cdot \omega_A \right) \wedge \omega_{t, \eps}^{n-1}}_{\ndef \mathrm{(V)}}
\end{align*}
for some constants $C, C' > 0$ sufficiently large.

\subsection*{Step~5: Computation of term~(II)}

This term can be handled in the same way as in~\cite[p.~258]{BochnerCGGN}.
Therefore
\[ \frac{ \mathrm{(II)} }{ |\sigma|^2 + 1 } \wedge \omega_{t, \eps}^{n-1} \le C \left( \beta_{j, \eps} + \frac{\eps^2}{|s_j|^2 + \eps^2} \cdot \omega_A \right) \wedge \omega_{t, \eps}^{n-1} \]
for some constant $C > 0$.
The right-hand side can be subsumed into term $\mathrm{(V)}$ by enlarging the constant $C'$.

\subsection*{Step~6: Taking the limit}

Putting everything together, we see from~\lref{main est} that
\begin{equation} \label{676}
0 \le \int_Y \frac{ \big\la i \, \Theta(\sE_Y(kF), h_{t, \eps}) \sigma, \, \sigma \big\ra }{ |\sigma|^2 + 1 } \wedge \omega_{t, \eps}^{n-1} \le \int_Y \mathrm{(III)} + t \cdot \mathrm{(IV)} + \mathrm{(V)}.
\end{equation}
We will first let $\eps \to 0$ (keeping $t$ fixed) and then let $t \to 0$.
To handle the first term on the right-hand side of~\lref{676}, define
\[ K(\sigma) \defn \liminf_{t \to 0} \liminf_{\eps \to 0} \int_Y \frac{|\sigma|_{t, \eps}^2}{|\sigma|_{t, \eps}^2 + 1} \omega_{t, \eps}^n, \]
where we have included the dependence of $|\sigma|$ on $t, \eps$ for clarification.

\begin{clm} \label{680}
If $\sigma \ne 0$, then $K(\sigma) > 0$.
\end{clm}

\begin{proof}
Let $U \Subset Y^\circ$ be a nonempty relatively compact open subset.
As $t, \eps \to 0$, we have smooth convergence $\omega_{t, \eps} \to \pi^* \omega_\ke$ and $h_{F, \eps} \to h_{F, 0}$ on $U$.
Therefore
\[ K(\sigma) \ge \lim_{t \to 0} \lim_{\eps \to 0} \int_U \frac{|\sigma|_{t, \eps}^2}{|\sigma|_{t, \eps}^2 + 1} \omega_{t, \eps}^n = \int_U \frac{|\sigma|^2}{|\sigma|^2 + 1} \pi^* \omega_\ke^n > 0, \]
where $|\sigma|$ is taken with respect to $\pi^* \omega_\ke$ and $h_{F, 0}$.
\end{proof}

Returning to~\lref{676}, the second term tends to zero:
\[ \lim_{t \to 0} \lim_{\eps \to 0} \int_Y t \cdot \mathrm{(IV)} = \lim_{t \to 0} t \cdot C (\pi^* K_X + tA)^{n-1} \cdot A = 0. \]
So does the third term, according to~\cite[p.~258]{BochnerCGGN}:
\[ \lim_{t \to 0} \lim_{\eps \to 0} \int_Y \mathrm{(V)} = 0. \]
In total, we get
\[ 0 \le \limsup_{t \to 0} \limsup_{\eps \to 0} \int_Y \frac{ \big\la i \, \Theta(\sE_Y(kF), h_{t, \eps}) \sigma, \, \sigma \big\ra }{ |\sigma|^2 + 1 } \wedge \omega_{t, \eps}^{n-1} \le \frac{q - p}n \cdot K(\sigma). \]

\subsection*{Step~7: Conclusion in case $p = q$}

In this case,
\[ \limsup_{t \to 0} \limsup_{\eps \to 0} \int_Y \frac{ \big\la i \, \Theta(\sE_Y(kF), h_{t, \eps}) \sigma, \, \sigma \big\ra }{ |\sigma|^2 + 1 } \wedge \omega_{t, \eps}^{n-1} = 0 \]
and we conclude as in~\cite[p.~258]{BochnerCGGN} that $\tau$ is parallel with respect to $\omega_\ke$.

\subsection*{Step~8: Conclusion in case $p > q$}

In this case, from $0 \le \frac{q - p}n \cdot K(\sigma)$ we deduce $K(\sigma) = 0$, hence $\sigma = 0$ by \cref{680}.
This is clearly equivalent to $\tau = 0$, which was to be shown. \qed

\section{The holonomy of the \kahler--Einstein metric on $X$}

In this section, we use the Bochner principle to obtain information about the holonomy group of $(\reg X, \omega_\ke)$.
It will turn out that it is a finite extension of a product of copies of $\U1$, cf.~\cref{Hol full polydisc}.

\begin{setup} \label{setup holonomy}
Let $(X, \omega_\ke)$ be as above.
Write $g_\ke$ for the associated Riemannian metric on $\reg X$ and $h_\ke$ for the associated Hermitian metric on $\T{\reg X}$.
Fix, once and for all, a smooth point $x \in \reg X$, and consider the tangent space $V \defn T_x X$ at that point.
We write
\begin{align*}
H \defn \operatorname{Hol}(\reg X, g_\ke, x) & \subset \U{V, h_{\ke, x}} \isom \U n \quad \text{and} \\
H^\circ \defn \operatorname{Hol}^\circ(\reg X, g_\ke, x) & \subset \U{V, h_{\ke, x}}
\end{align*}
for the corresponding (restricted) holonomy group.
Recall from~\cite[Cor.~10.41]{Besse87} that the action $H^\circ \acts V$ is \emph{totally decomposed}.
That is, there are decompositions
\begin{equation} \label{canonical}
V = V_0 \oplus V_1 \oplus \cdots \oplus V_m \quad \text{and} \quad H^\circ = H_1^\circ \x \cdots \x H_m^\circ
\end{equation}
such that for each $1 \le i \le m$, the factor $H_i^\circ$ acts irreducibly and non-trivially on~$V_i$ and trivially on $V_j$ for $j \ne i$.
Set $n_i \defn \dim_\C V_i$ for each $0 \le i \le m$.
\end{setup}

\begin{lem}
In the above setup, we always have $n_0 = 0$.
\end{lem}

\begin{proof}
Assuming that $n_0 \ge 1$, apply~\cite[Thm.~10.38]{Besse87} to a sufficiently small simply connected neighborhood $x \in U \subset X$.
We obtain that $(\reg X, \omega_\ke)$ is locally a product containing a non-trivial flat factor.
In particular, that factor is Ricci-flat.
This contradicts the fact that $\Ric(\omega_\ke) = - \omega_\ke$ is negative definite on $\reg X$.
\end{proof}

The following lemma can be found in~\cite[Lemma~10.113]{Besse87}, but without a proof.

\begin{lem} \label{normalizer lemma}
The quotient of normalizer subgroups
\[ \factor{N_{\U V}(H^\circ)}{N_{\U{V_1}}(H_1^\circ) \x \cdots \x N_{\U{V_m}}(H_m^\circ)} \]
is finite, of order $\le m!$.
\end{lem}

\begin{proof}
For simplicity of notation, let us assume that $m = 2$.
The general case can be handled similarly.
We will show that any $g \in \U V$ that normalizes $H^\circ$ must permute the summands of $V$, i.e.~either $g(V_i) = V_i$ for $i = 1, 2$ or $g(V_i) = V_{3-i}$ for $i = 1, 2$.
This implies the claim.

To this end, consider each $V_i$ as an (irreducible) $H^\circ$-representation via the projection map $H^\circ \surj H_i^\circ$.
Then clearly, $V \isom V_1 \oplus V_2$ as $H^\circ$-representations.
But note that $V_1$ and $V_2$ are \emph{not} isomorphic as $H^\circ$-representations, even if $n_1 = n_2$ and $H_1^\circ = H_2^\circ$ as subgroups of $\U{n_1}$.
This is because the respective kernels are $\set1 \x H_2^\circ$ and $H_1^\circ \x \set1$, which are never equal.
By Schur's lemma, the only non-trivial subrepresentations of $V$ are $V_1 \oplus 0$ and $0 \oplus V_2$.

Arguing by contradiction, assume now that $g(V_1) \ne V_1, V_2$.
By the above, $g(V_1)$ is not stable under the action of $H^\circ$.
So there is an $h \in H^\circ$ such that $h \big( g(V_1) \big) \ne g(V_1)$.
In other words, $(g\inv hg)(V_1) \ne V_1$.
But then $g\inv hg \not\in H^\circ$, and $g$ does not normalize $H^\circ$.

The same argument also shows that $g(V_2) \in \set{ V_1, V_2 }$.
Since $g$ acts in particular as a bijection, it must permute $V_1$ and $V_2$.
\end{proof}

\begin{lem} \label{Hol0 polydisc}
Assume that $X$ satisfies condition~\lref{main1.3}, i.e.~there is a semispecial tensor $\psi$ on $X$ such that the hypersurface $\set{ \psi_x = 0 } \subset \bP V$ is reduced.
Then the restricted holonomy of $X$ is $H^\circ = \U 1^n$.
More precisely, in~\lref{canonical} we have $m = n$ and $n_i = 1$, $H_i^\circ = \U{V_i} \isom \U 1$ for each $1 \le i \le m$.
\end{lem}

The proof is modeled on~\cite[proof of Thm.~1.1]{CataneseDiScala13}.
They argue on the universal cover of $X$, which we clearly cannot do.
Our observation is that much of the argument can be recast in purely local terms.

\begin{proof}
Consider the decomposition
\[ V = U_1 \oplus U_2 \]
where $U_1$ is the sum of all the $V_i$ such that the corresponding $H_i^\circ$ is the holonomy of a locally irreducible Hermitian symmetric space, and $U_2$ is the sum of the remaining factors.

By \cref{bochner cor}, the tensor $\psi$ is parallel with respect to $\omega_\ke$.
Then by the holonomy principle, the homogeneous degree $n$ polynomial $\psi_x$ on $V$ is $H^\circ$-invariant.
In particular, its zero scheme $\set{ \psi_x = 0 } \subset \bP V$ is also $H^\circ$-invariant and we can apply~\cite[Prop.~A.1]{CataneseFranciosi09} to it.
We obtain that $\psi_x$ does not depend on the summand~$U_2$, i.e.~there is a polynomial $f$ on $U_1$ (likewise homogeneous of degree $n$) such that $\psi_x$ is the pullback of $f$ along the projection $V \surj U_1$.

Pick a bounded symmetric domain $0 \in \cD \subset \C^{\dim U_1}$ such that the action of $H^\circ$ on $U_1$ equals the action of $K^\circ$ on $T_0 \cD$, where $K \subset \Aut(\cD)$ is the stabilizer of $0 \in \cD$ and $K^\circ$ is its identity component.
Now apply~\cite[Cor.~2.2 and remark thereafter]{CataneseDiScala13} to the polynomial $f$ considered as a function on $T_0 \cD$ (and hence on $\cD$).
We obtain a splitting
\[ f = c \cdot \prod_{j=1}^p N_j^{k_j}, \]
where:
\begin{itemize}
\item $c \ne 0$ is a suitable constant,
\item $\cD = \cD_1' \x \cdots \x \cD_p' \x \cD''$, where the $\cD_i'$ are irreducible and of tube type, while $\cD''$ has no irreducible factor of tube type,
\item each $N_j$ is a $K$-semi-invariant polynomial on $\cD_j'$, of degree equal to $r_j \defn \rk \cD_j'$, and
\item all the exponents $k_j$ satisfy $0 \le k_j \le 1$ since $f$ is squarefree, the hypersurface $\set{ \psi_x = 0 }$ being reduced.
\end{itemize}
It then holds that
\[ \sum_{j=1}^p k_j r_j = \deg f = n \ge \dim \cD = \sum_{j=1}^p \dim \cD_j' + \dim \cD'' \ge \sum_{j=1}^p r_j + \dim \cD'' \]
while at the same time obviously
\[ \sum_{j=1}^p k_j r_j \le \sum_{j=1}^p r_j + \dim \cD''. \]
Therefore we must have equality everywhere.
We conclude that $k_1 = \cdots = k_p = 1$, $\dim \cD'' = 0$, and $U_2 = 0$.
We also conclude that $r_j = \dim \cD_j'$ for each $j$, so each $\cD_j'$ is a polydisc.
But the $\cD_j'$ are irreducible, hence isomorphic to the (one-dimensional) unit disc.
In particular, $p = n$.
This proves the claim.
\end{proof}

\begin{cor} \label{Hol full polydisc}
Assume that $X$ satisfies condition~\lref{main1.3}.
Then the full holonomy group $H$ is compact.
\end{cor}

\begin{proof}
By \cref{Hol0 polydisc}, we have $H_i^\circ = \U{V_i}$ and in particular $N_{\U{V_i}}(H_i^\circ) = H_i^\circ$ for each $1 \le i \le n$.
\cref{normalizer lemma} then shows that the normalizer $N(H^\circ)$ is in fact a finite extension of $H^\circ$.
But the identity component $H^\circ \subset H$ is a normal subgroup, so $H \subset N(H^\circ)$.
Therefore also $H$ is a finite extension of $H^\circ$, hence compact.
\end{proof}

\section{Proof of \cref{main1}}

\cref{main1} follows immediately from the following more precise and slightly stronger result.

\begin{thm} \label{main1'}
Let $X$ be a normal irreducible compact complex space of dimension~$n$.
Then, the following statements are equivalent:
\begin{enumerate}
\item\label{main1'.1} We have $X \isom \factor{\H n}\Gamma$, where $\Gamma \subset \Aut(\H n)$ is a discrete cocompact subgroup whose action is free in codimension one.
\item\label{main1'.2} We have $X \isom \factor YG$, where $Y$ is a compact complex manifold whose universal cover $\wt Y$ is isomorphic to $\H n$ and $G \subset \Aut(Y)$ is a finite subgroup whose action is free in codimension one.
\item\label{main1'.3} The space $X$ is a projective variety with klt singularities, the canonical divisor $K_X$ is ample, and there is a semispecial tensor $\psi$ on $X$ such that for some (equivalently, any) point $p \in \reg X$, the hypersurface $\set{ \psi_p = 0 } \subset \bP(T_p X)$ is reduced.
\end{enumerate}
\end{thm}

\begin{proof}
``\lref{main1'.1} $\imp$ \lref{main1'.2}'': This is essentially a consequence of Selberg's lemma, cf.~\cite[proof of implication ``(1.3.1) $\imp$ (1.3.2)'']{GKPT1_ENS}.
Note in particular that said proof and all its references continue to work verbatim if the unit ball $\bB^n$ is replaced by the polydisc $\H n$ everywhere.

``\lref{main1'.2} $\imp$ \lref{main1'.3}'':
By~\cite[Ch.~3, Thm.~8.4]{MorrowKodaira71}, $K_Y$ is ample and in particular $Y$ is projective.
Since $f \from Y \to X$ is a finite \qe Galois cover, it follows that $X$ is projective with klt singularities and that $K_Y = f^* K_X$.
Therefore also $K_X$ is ample.

To obtain a semispecial tensor $\psi$ on $X$, we first remark that~\lref{main1'.1} and~\lref{main1'.2} are actually equivalent.
This is because the composed map
\[ \H n \isom \wt Y \lto Y \lto X \]
exhibits $X$ as a quotient of $\H n$ by a discrete cocompact subgroup $\Gamma \subset \Aut(\H n)$ acting freely in codimension one.
See~\cite[proof of implication ``(1.3.3) $\imp$ (1.3.1)'']{GKPT1_ENS} for more details, and note again that replacing the unit ball by the polydisc does not affect the validity of the argument.

Now consider the following tensor $\wt\psi$ on $\H n$ with coordinates $z_1, \dots, z_n$:
\[ \wt\psi \defn \frac{\d z_1 \cdots \d z_n}{\d z_1 \wedge \cdots \wedge \d z_n} \in \HH0.\H n.\Sym^n(\Omegap{\H n}1)(-K_{\H n})., \]
which is just a suggestive notation for the homomorphism $\O{\H n}(K_{\H n}) \to \Sym^n(\Omegap{\H n}1)$ sending $\d z_1 \wedge \cdots \wedge \d z_n \mapsto \d z_1 \cdots \d z_n$.
Recall from~\cite[Cor.~on p.~167]{Rudin69} that the automorphism group $\Aut(\H n)$ is the semidirect product $\Aut(\bH)^n \rtimes \frS_n$.
More precisely,  there is a split short exact sequence
\[ 1 \lto \Aut(\bH)^n \lto \Aut(\H n) \lto \frS_n \lto 1, \]
where $\frS_n$ is the symmetric group on $n$ letters.
Clearly, $\wt\psi$ is invariant under the index two subgroup
\[ \Aut(\bH)^n \rtimes \frA_n \subset \Aut(\H n), \]
where $\frA_n \subset \frS_n$ denotes the alternating group.
We distinguish two cases: if $\Gamma$ is contained in $\Aut(\bH)^n \rtimes \frA_n$, then $\wt\psi$ directly descends to a semispecial tensor $\psi$ on $X$, with $\eta = \O X$.
Otherwise, the quotient of $\H n$ by $\Gamma' \defn \Gamma \cap (\Aut(\bH)^n \rtimes \frA_n)$ yields a \qe double cover $\sigma \from X' \to X$ such that there is a factorization
\[ \H n \lto X' \xrightarrow{\;\;\sigma\;\;} X \]
and $\wt\psi$ descends to a semispecial tensor $\psi' \in \HH0.X'.\Sym^n(\Omegap{X'}1)(-K_{X'}).$.
Write $\sigma_* \O{X'} = \O X \oplus \eta$, where $\eta$ is a rank one reflexive sheaf with $\eta^{[2]} \isom \O X$.
Then we have an isomorphism
\[ \sigma_* \big( \!\Sym^n(\Omegap{X'}1)(-K_{X'}) \big) = \Sym^n(\Omegap X1)(-K_X) \oplus \big( \!\Sym^n(\Omegap X1)(-K_X) [\tensor] \eta \big) \]
obtained from the projection formula on $\reg X$ and extended to all of $X$ by reflexivity.
Under this isomorphism, $\psi'$ corresponds to a global section $\psi$ of the second summand on the right-hand side, which again is a semispecial tensor on $X$.

In both cases, the hypersurface $\set{ \psi_p = 0 } \subset \bP(T_p X)$ is reduced for any $p \in \reg X$ because $\wt\psi$ has this property.
The fact that reducedness at some smooth point implies reducedness at any smooth point is a consequence of~\lref{bochner cor.2}.

``\lref{main1'.3} $\imp$ \lref{main1'.1}'': Consider the \kahler--Einstein metric $\omega_\ke$ as in \cref{bochner setup}.
By \cref{Hol0 polydisc} and \cref{Hol full polydisc}, the holonomy group $H$ of $(\reg X, \omega_\ke)$ is a finite extension of ${\U1}^n$.
Let $X' \to X$ be the finite \qe Galois cover corresponding to the finite quotient $\pi_1(\reg X) \surj \factor{H}{H^\circ}$.
Then the holonomy of $(\reg X', \omega_{\ke, X'})$ is isomorphic to ${\U1}^n$.
In particular, the tangent bundle $\T{\reg X'}$ splits as a direct sum of line bundles.
By~\cite[Thm.~5.1]{Patel23} or~\cite[Cor.~1.5]{GrafPatel24}, there is a finite \qe Galois cover $Y \to X'$ such that $Y$ is a projective manifold whose universal cover $\wt Y$ is $\H n$.
After taking Galois closure~\cite[Lemma~2.8]{BochnerCGGN}, we may assume that $f \from Y \to X$ is also Galois.

This argument already shows that ``\lref{main1'.3} $\imp$ \lref{main1'.2}'' holds.
Since we have remarked above that~\lref{main1'.1} and~\lref{main1'.2} are equivalent, the desired implication is proven.
\end{proof}

\section{Proof of corollaries}

\begin{proof}[Proof of \cref{cor gen type}]
By~\cite{BCHM10}, we can run the MMP on $X$ and obtain a birational contraction to the canonical model $\phi \from X \dashrightarrow X_\canmod$, where $K_{X_\canmod}$ is ample and $X_\canmod$ has canonical (in particular, klt) singularities.
Since $\phi$ is a contraction (meaning that $\phi\inv$ does not contract any divisors), the given semispecial tensor $\psi$ on $X$ induces a semispecial tensor $\psi' \defn \phi_*(\psi)$ on $X_\canmod$.
We must show that $\psi'$ defines a reduced hypersurface at some point of $X_\canmod$, for then we can conclude by \cref{main1'}.

\begin{clm}[Reducedness is an open property] \label{reduced is open}
The set
\[ U \defn \set{ x \in X \;\;\big|\;\; \text{$\psi_x \ne 0$ and $Z(\psi_x) \subset \bP(T_x X)$ is reduced} } \subset X \]
is Zariski-open in $X$.
\end{clm}

We give two proofs of \cref{reduced is open}: the first one is more elementary in nature, while the second one uses modern algebraic geometry machinery.

\begin{proof}[First proof of \cref{reduced is open}]
Let $P_d$ be the projective space of homogeneous degree $d$ polynomials in $n$ variables (with complex coefficients).
Consider the map
\[ \phi_d \from P_d \x P_{n - 2d} \lto P_n \]
given by sending $(g, h) \mapsto g^2 h$.
The image $N_d$ of the proper map $\phi_d$ is closed and the locus $R_n \subset P_n$ of reduced polynomials is the complement of the union $N_1 \cup \cdots \cup N_{\rd n/2.}$, hence open.
Sending $x \mapsto \psi_x$ defines a rational map $X \dashrightarrow P_n$ (at least locally, after choosing a local trivialization of $\T X$) and $U$ is the preimage of $R_n$ under this map.
\end{proof}

\begin{proof}[Second proof of \cref{reduced is open}]
Regarding $\psi$ as a map from the total space of $\T X$ to the total space of $\O X(-K_X) \tensor \eta$, we may consider the scheme-theoretic preimage of the zero section and then projectivize it to obtain a subscheme $Z \subset \bP(\T X)$.
Set $V \defn \set{ x \in X \;\big|\; \psi_x \ne 0 } \subset X$.
Then the base change $Z_V \to V$ is (at least locally, after choosing a local trivialization of $\T X$) a family of degree~$n$ hypersurfaces in~$\P{n-1}$.
Since the Hilbert polynomial of a hypersurface depends only on its degree and dimension, $Z_V \to V$ is therefore flat by~\cite[Ch.~III, Thm.~9.9]{Har77}.
Note that~$U$ is the set of points over which the fibre of $Z_V \to V$ is reduced.
This set is open by~\cite[Tag~0C0E]{stacks-project}.
\end{proof}

Back to the proof of \cref{cor gen type}: recall that by assumption, there is a point $p \in X$ such that the hypersurface $Z(\psi_p)$ is reduced.
By \cref{reduced is open}, we may assume that $p$ is contained in the locus where $\phi$ is an isomorphism.
It is then clear that $\psi'$ defines a reduced hypersurface at $\phi(p)$, ending the proof.
\end{proof}

\begin{rem} \label{rem unif}
In the situation of \cref{cor gen type}, we would like to say something about the universal cover of $X$ itself and not just about $X_\canmod$.
However, all we can say is the following:
\begin{itemize}
\item There is an open subset $U \subset X$ whose universal cover $\wt U$ is bimeromorphic to a big open subset of the polydisc.
(Here, ``big'' means that the complement has codimension at least two.)
\item We can blow up $X$ to a projective variety $X_2$ with quotient singularities which is of the form $\factor{\wt Y}{\Gamma}$, where $\wt Y$ is a complex manifold bimeromorphic to the polydisc and $\Gamma$ is the same group as in \cref{cor gen type}.
\end{itemize}
For the first statement, set $U \defn X \setminus \Exc(\phi)$.
Then $\phi(U) \subset X_\canmod$ is a big and smooth open subset, and it is isomorphic to $U$.
If $\pi \from \H n \to X_\canmod$ is the quotient map, then $\pi\inv \big( \phi(U) \big)$ is a big open subset of the polydisc, hence simply connected.
Since $\pi$ is \'etale over the smooth open subset $\phi(U)$, this shows that $\pi\inv \big( \phi(U) \big) \to \phi(U)$ is the universal cover of $\phi(U)$.
For the second statement, consider a resolution of indeterminacy
\[ \begin{tikzcd}
& X_1 \arrow[dl] \arrow[dr] \\
X \arrow[rr, dashed, "\phi"] & & X_\canmod
\end{tikzcd} \]
and the normalized fibre product
\[ \begin{tikzcd}
Y \arrow[d, "/\Gamma"'] \arrow[rr, "\text{bimerom.}"] & & \H n \arrow[d, "/\Gamma"] \\
X_1 \arrow[rr] & & X_\canmod.
\end{tikzcd} \]
Then $Y$ is a normal complex space on which $\Gamma$ acts, and $Y \to X_1 = \factor Y\Gamma$ is the quotient map.
Let $\wt Y \to Y$ be the functorial resolution~\cite[Thm.~3.45]{Kol07}.
The action of $\Gamma$ on $Y$ lifts to $\wt Y$, with quotient $X_2 \defn \factor{\wt Y}\Gamma$.
There is a commutative diagram
\[ \begin{tikzcd}
\wt Y \arrow[d, "/\Gamma"'] \arrow[rr, "\text{bimerom.}"] & & Y \arrow[d, "/\Gamma"] \\
X_2 \arrow[rr] & & X_1.
\end{tikzcd} \]
This proves the second statement.
In general, $\Gamma$ may not act freely on $\wt Y$ and so we do not get a statement about the universal cover of $X_2$.
However, there is a divisor~$\Delta$ on $X_2$ such that $\wt Y$ is the \emph{orbifold} universal cover of $(X_2, \Delta)$.
(For the definition of orbifold universal cover, see~\cite[Def.~24]{MYeq}.)
In fact, set $\Delta = \sum_i \big( 1 - \frac1{m_i} \big) \Delta_i$, where $m_i$ is the order of ramification of the map $\wt Y \to X_2$ over a general point of~$\Delta_i$.
\end{rem}

\begin{proof}[Proof of \cref{cor conjugate}]
This follows immediately from \cref{main1'} once we note that all of the following properties are invariant under conjugation by $\sigma$: having klt singularities, having ample canonical divisor, and carrying a semispecial tensor with reduced hypersurface.
See~\cite[Sec.~5]{CataneseDiScala13} for more details.
\end{proof}

\begin{proof}[Proof of \cref{cor dim 3}]
The proof of ``\lref{dim3.1} $\imp$ \lref{dim3.2}'' is exactly the same as in the situation of \cref{main1}, and is therefore omitted.
For the converse, using the notation from \cref{setup holonomy} we only have to show that $H^\circ = {\U1}^3$, for then the rest of the argument goes through.

We follow the proof of~\cite[Thm.~1.9]{CataneseFranciosi09}.
They make a case distinction according to the structure of the cubic curve $\set{ \psi_p = 0 } \subset \bP(T_p X)$.
In cases (a)--(f) (using their notation), their arguments work verbatim for us.
Therefore we concentrate on case (g), which is the case that $\set{ \psi_p = 0 } = 3L$ is a triple line.
We will show that this case cannot occur.

Let $\sL^\circ \subset \Omegap{\reg X}1$ be the line subbundle corresponding to the \emph{reduced} zero locus of $\psi$ considered as a subset of $\bP(\T X)$.
Extend $\sL^\circ$ to a Weil divisorial (i.e.~rank one reflexive) sheaf $\sL \subset \Omegar X1$.
We obtain an inclusion
\[ \sL^{[3]}(-K_X) [\tensor] \eta \; \subset \; \Sym^3(\Omegap X1)(-K_X) [\tensor] \eta \]
and the given semispecial tensor $\psi$ is contained in the left-hand side, which therefore has a global section.
This means that $\O X(K_X) [\tensor] \eta \subset \sL^{[3]}$, so that $\sL$ is big because $K_X$ is ample.
But this contradicts $\sL \subset \Omegar X1$ by~\cite[Cor.~1.3]{Gra12}.
\end{proof}

\end{document}